\documentclass[12pt]{article}
\usepackage[a4paper]{geometry}

\usepackage[all,cmtip]{xy}

\usepackage{mathptmx}
\usepackage[scaled=.90]{helvet}
\usepackage{courier}

\usepackage{amssymb, amsthm, amsmath, stmaryrd}

\newcommand{\ch}{\mathrm{char}}

\newcommand{\rmB}{\mathrm{B}}
\newcommand{\rmH}{\mathrm{H}}
\newcommand{\rmS}{\mathrm{S}}
\newcommand{\rmC}{\mathrm{C}}
\newcommand{\rmK}{\mathrm{K}}
\newcommand{\rmQ}{\mathrm{Q}}

\newcommand{\rme}{\mathrm{e}}

\newcommand{\bbF}{\mathbb{F}}
\newcommand{\bbP}{\mathbb{P}}

\newcommand{\bbS}{\mathbb{S}}
\newcommand{\bbZ}{\mathbb{Z}}

\newcommand{\modmod}{/\!\!/}
\newcommand{\bbSp}{\bbS\modmod p}
\newcommand{\bbSt}{\bbS\modmod 2}

\newcommand{\calE}{\mathcal{E}}
\newcommand{\calS}{\mathcal{S}}

\newcommand{\UmodO}{\mathrm{U}/\mathrm{O}}
\newcommand{\bo}{\mathrm{bo}}
\newcommand{\BO}{\mathrm{BO}}

\newcommand{\moravak}{\mathrm{k}}

\newcommand{\K}{\mathrm{K}}
\newcommand{\M}{\mathrm{M}}
\newcommand{\MO}{\mathrm{MO}}

\newcommand{\GL}{\mathrm{GL}}

\newcommand{\ev}{\mathrm{ev}}
\newcommand{\cof}{\mathrm{cof}}

\newcommand{\Tor}{\mathrm{Tor}}

\newcommand{\THH}{\mathrm{THH}}
\newcommand{\TAQ}{\mathrm{TAQ}}

\usepackage[pdftex]{hyperref}

\hypersetup{
  colorlinks = true,
  linkcolor = black,
  urlcolor = blue,
  citecolor = black,
}

\newtheorem{proposition}{Proposition}[section]

\newtheorem{theorem}[proposition]{Theorem}
\newtheorem{corollary}[proposition]{Corollary}

\theoremstyle{definition}
\newtheorem{remark}[proposition]{Remark}
\newtheorem{definition}[proposition]{Definition}

\newtheorem{example}[proposition]{Example}

\numberwithin{equation}{section}

\begin{document}

\title{Commutative~\texorpdfstring{$\bbS$}{S}-algebras of prime characteristics\\ and applications to unoriented bordism}

\author{Markus Szymik}

\date{February 2014}

\maketitle


\renewcommand{\abstractname}{}

\begin{abstract}\noindent
  The notion of highly structured ring spectra of prime characteristic is made precise and is studied via the versal examples~$\bbSp$ for prime numbers~$p$. These can be realized as Thom spectra, and therefore relate to other Thom spectra such as the unoriented bordism spectrum~$\MO$. We compute the Hochschild and Andr\'e-Quillen invariants of the~$\bbSp$. Among other applications, we show that~$\bbSp$ is not a commutative algebra over the Eilenberg-Mac Lane spectrum~$\rmH\bbF_p$, although the converse is clearly true, and that~$\MO$ is not a polynomial algebra over~$\bbSt$. 
  
  \phantom{}
  
  
\end{abstract}


\section*{Introduction}

In the world of ordinary rings, those of prime characteristic are of special importance, and their algebras encode much of the elementary arithmetic of the ring~$\bbZ$ of integers. Let us recall: If~$A$ is a commutative ring, and~$p$ is a prime number, then~$A$ has characteristic~$p$~(written~$\ch(A)=p$) if  the image of~$p$ under the unit~$\bbZ\to A$ is zero. Equivalently, there is a unique ring map~\hbox{$\bbF_p\to A$} from the prime field~$\bbF_p$ with~$p$ elements. The aim of this writing is to generalize and explore this notion of prime characteristic from the setting of ordinary commutative rings in the context of commutative ring spectra in the highly structured sense:~$E_\infty$ ring spectra, or commutative~$\bbS$-algebras in more recent terminology. This can be understood as an attempt to unveil some of the arithmetic surrounding the sphere spectrum~$\bbS$.

There might be more than one way to achieve such a generalization. Clearly, the underlying ring~$\pi_0(E)$ of a commutative~$\bbS$-algebra~$E$ of characteristic~$p$ should be of characteristic~$p$ (in the ordinary sense), and the Eilenberg-Mac Lane spectrum~$\rmH A$ of an ordinary ring~$A$ of characteristic~$p$ should be a commutative~$\bbS$-algebra of characteristic~$p$. This is fulfilled in the present approach. 

If~$E$ is a commutative ring spectrum (up to homotopy or~$E_\infty$), then we will say that~$E$ has characteristic~$p$ and write~$\ch(E)=p$ if~$p$ is mapped to zero under the ring map~\hbox{$\bbZ=\pi_0(\bbS)\to\pi_0(E)$} induced by the unit~$\bbS\to E$ from the sphere spectrum~$\bbS$. In Section~\ref{sec:first_examples} we briefly review the known results about commutative ring spectra of prime characteristic. In Section~\ref{sec:SSp} we define versal examples---commutative~$\bbS$-algebras~$\bbSp$~(one for each prime number~$p$) such that the category of commutative~$\bbSp$-algebras is an appropriate place to study commutative~$\bbS$-algebras of characteristic~$p$. Section~\ref{sec:homology_and_homotopy} contains some homology and homotopy computations which are necessary for the later applications. Section~\ref{sec:Thom} contains a description of all the versal examples~$\bbSp$ as~$E_\infty$ Thom spectra, see Theorems~\ref{thm:Thom_2} and~\ref{thm:Thom_odd}. This relates the spectra~$\bbSp$ to other Thom spectra, and it also enables us to describe the topological Hochschild and Andr\'e-Quillen invariants of the spectra~$\bbSp$. 

The final Section~\ref{sec:applications} contains various applications with an emphasis on the unoriented bordism spectrum~$\MO$. While~$\MO$ is not an algebra under the Eilenberg-Mac Lane spectrum~$\rmH\bbF_2$, it is an algebra under~$\bbSt$. This in turn implies that~$\bbSt$ is not an~$\rmH\bbF_2$-algebra. More generally, we are able to show that~$\bbSp$ is not an~$\rmH\bbF_p$-algebra for any prime~$p$, see Theorem~\ref{thm:no Eoo map}, although the converse is clearly true. It is shown in~\cite{Szymik} that higher bordism spectra, such as Spin and String bordism, can be treated analogously, once one is willing to work chromatically, and once one has set up a theory of characteristics in that context. 


\section*{Conventions}

Throughout the text, commutative~$\bbS$-algebras will be used as the chosen model for ring spectra with an~$E_\infty$ multiplication, see~\cite{EKMM}, in particular Chapter~VII. The category of~$\bbS$-modules has a Quillen model structure such that all objects are fibrant. The cofibrations are the retracts of relative cell~$\bbS$-modules. If~$R$ is a commutative~$\bbS$-algebra~(such as~$R=\bbS$ and later~$R=\bbSp$), then the category of commutative~$R$-algebras has a model structure where the equivalences and fibrations are created on underlying spectra. The cofibrations are the retracts of relative cell commutative~$R$-algebras. This has the effect that~$R$ is always cofibrant as a commutative~$R$-algebra, even if is is not cofibrant as a spectrum.

There are by now various other models for structured ring spectra, most of them discussed and shown to be equivalent in~\cite{MMSS}, and each of them serves our purposes equally well. We will also continue to employ the more generic~$E_\infty$ terminology to emphasize this fact. The notation~$\calE_\infty$, and~$\calS_\infty$, will be used for the category of~$\calE_\infty$ ring spectra/commutative~$\bbS$-algebras, and the category of~$\bbS$-modules/spectra, respectively. We write~$\bbS$ for the sphere spectrum as a commutative~$\bbS$-algebra, and~$\rmS^n=\Sigma^n\bbS$ for the suspension spectra. The notation~$\rmS^n$ will also be used for the usual euclidean spheres. We will sometimes abbreviate~$\rmH\bbF_p$ to~$\rmH$ when the prime is clear from the context. Also, unless otherwise specified, all rings, algebras, ring spectra, and algebra spectra are assumed to be commutative from now on. While this will be our default, we may nevertheless use the word~`commutative' for emphasis.


\section{Examples and counterexamples}\label{sec:first_examples}

In this section, we will recall some known results on commutative ring spectra~$E$ with~$\ch(E)=p$ in the following sense.

\begin{definition}
If~$E$ is a commutative ring spectrum (up to homotopy or~$E_\infty$), and~$p$ is a prime number, then we will say that~$E$ has characteristic~$p$, if the ordinary commutative ring~$\pi_0(E)$ has characteristic~$p$ in the usual sense. We will also use the notation~\hbox{$\ch(E)=p$}.
\end{definition}

\begin{remark}
In the~$E_\infty$ setting, it may be worthwhile noting that any cofibrant replacement of a commutative~$\bbS$-algebra of characteristic~$p$ has characteristic~$p$ as well: An equivalence~$E^\cof\to E$ of commutative~$\bbS$-algebras is an equivalence of underlying spectra, so that~$p\in\pi_0\bbS$ goes to~$0\in\pi_0E^\cof$ under the unit~$\bbS\to E^\cof$ if and only if it does so in~$\pi_0E$.
\end{remark}

We can now discuss some examples and counterexamples: graded Eilenberg-Mac Lane spectra, Morava K-theory spectra, and Moore spectra. See also~\cite{rudyak} for a treatment of some of the
topics discussed here.

\subsection{Graded Eilenberg-Mac Lane spectra}

We will say that a spectrum~$E$ is {\it additively a graded
  Eilenberg-Mac Lane spectrum} if it is equivalent as a spectrum to
$\rmH M_*$ for some graded group~$M_*$. We will say that a ring
spectrum~$E$ (up to homotopy) is {\it multiplicatively a graded
  Eilenberg-Mac Lane spectrum} if it is equivalent as a commutative ring spectrum
(up to homotopy) to~$\rmH A_*$ for some graded commutative ring~$A_*$.

Note that~$\rmH A_*$ is not only a ring spectrum up to homotopy, but has a preferred~$E_\infty$ structure, Boardman's multiplication~\cite{boardman}. In fact, Richter has shown more generally that for a differential graded commutative algebra~$A_*$, the graded Eilenberg-MacLane spectrum~$\rmH A_*$ is an~$E_\infty$-monoid in the category of~$\rmH\bbZ$-module spectra. See~\cite[Proposition~6.1]{Richter:Symmetry}, and~\cite[Theorem 5.6.1]{Richter:Homotopy}. Here, we only need the case where the differential is trivial, so that we have a graded commutative ring. Then the forgetful functor from~$\rmH\bbZ$-modules to~$\bbS$-modules gives rise to an~$E_\infty$ ring spectrum~$\rmH A_*$ that is a commutative~$\rmH A_0$-algebra. This~$E_\infty$ structure is essentially unique if~$A_*$ is concentrated in dimension~$0$, so that~$\rmH A_*$ is discrete. We will see, in Theorem~\ref{thm:non-uniqueness-for-structures}, that it is not unique in general.

\subsection{Results for the mod-\texorpdfstring{$p$}{p} case}

Let us start with a result which reduces the more difficult
multiplicative question to the easier additive question.

\begin{theorem} {\upshape(\cite[Theorem 1.1]{boardman})}
\label{thm:additive_implies_multiplicative}
  Suppose~$E$ is a commutative ring spectrum (up to homotopy) which additively is
  a graded Eilenberg-Mac Lane spectrum, and~$\ch(E)=p$ for some prime number~$p$. Then~$E$
  is multiplicatively a graded Eilenberg-Mac Lane spectrum (up to
  homotopy).
\end{theorem}

The following result shows that in the case~$p=2$ the Eilenberg-Mac Lane hypo\-thesis is superfluous: there are no other examples.

\begin{theorem}{\upshape(\cite[Theorem 1.1]{würgler},
  \cite{pazhitnov+rudyak})}\label{thm:mod2isEM}
  Suppose~$E$ is a commutative ring spectrum (up to homotopy) with~$\ch(E)=2$. Then~$E$ is multiplicatively an
  Eilenberg-Mac Lane spectrum (up to homotopy).
\end{theorem}

Versions of this theorem are also attributed to unpublished work of Hopkins and Mahowald, see~\cite[Theorem 5]{yan:proc},~\cite[Theorem 5.1]{yan:trans} as well as~\cite[Theo\-rem~IX.5.5]{rudyak}, for example.

As mentioned above, we will see later, in Theorem~\ref{thm:non-uniqueness-for-structures}, that both theorems become false when the weak up to homotopy notion is replaced by the strong~$E_\infty$ notion of a ring spectrum. In this latter setting, using Dyer-Lashof operations, Steinberger has obtained the following result.

\begin{theorem}{\upshape(\cite[III.4.1]{steinberger})}\label{thm:steinberger:splitting}
  If~$E$ is a commutative~$\bbS$-algebra of finite type with~\hbox{$\pi_0(E)=\bbF_p$} for some prime number~$p$, then~$E$ is
  additively a graded Eilenberg-Mac Lane spectrum.
\end{theorem}


\subsection{Results for the~\texorpdfstring{$p$}{p}-local case}

Let us add some results on the~$p$-local situation. These will not
be used in the following. For odd primes, the situation is fairly
rigid.

\begin{theorem}{\upshape(\cite[Theorem 1.2]{boardman})}
  Let~$p$ be an odd prime. Suppose the ring spectrum~$E$ is additively
  a graded Eilenberg-Mac Lane spectrum with~$\pi_*(E)$ a free module
  over~$\bbZ_{(p)}$. Then~$E$ is multiplicatively a graded
  Eilenberg-Mac Lane spectrum.
\end{theorem}

For the even prime, there are exotic examples. 

\begin{theorem}{\upshape(\cite[Theorem 1.3]{boardman})} There exist
  ring spectra up to homotopy~$E$ with~$\pi_*(E)$ a
  free~$\bbZ_{(2)}$-module and which are additively graded
  Eilenberg-Mac Lane spectra, but not multiplicatively.
\end{theorem} 

Boardman's result for the prime 2 is complemented by the following
result of Astey's, which characterizes the 2-local ring spectra which
are graded Eilenberg-Mac Lane spectra. It involves a~$3$-cell complex which is build using the stable Hopf map~$\eta\colon\rmS^1\to\rmS^0$ as well as another map~$t$ in the homotopy~$\pi_2(\bbS_{(2)}\cup_\eta\rme^2)\cong\bbZ_{(2)}$ of its cone: The collapse map induces an injection~$\pi_2(\bbS_{(2)}\cup_\eta\rme^2)\to\pi_2(\rmS^2)=\bbZ$ with image consisting of the subgroup of even numbers, and~$t$ is defined as the (unique) pre-image of~$2$. The cone~$I=(\bbS_{(2)}\cup_\eta\rme^2)\cup_t\rme^3$ is often called an inverted question mark complex---from the point of view of the action of the Steenrod algebra on its mod~$2$ cohomology.

\begin{theorem}{\upshape(\cite[Theorem 1.2]{astey})}
  A~$2$-local ring spectrum up to homotopy~$E$ is additively a
  graded Eilenberg-Mac Lane spectrum if and only if the unit map
  extends over the inverted question mark complex~$I$.
\end{theorem}

Again, there are also results for~commutative~$\bbS$-algebras~$E$ that satisfy~\hbox{$\pi_0(E)=\bbZ_{(p)}$}, see~\cite[III.4.2 and III.4.3]{steinberger}.


\subsection{Morava K-theory spectra}

Recall that the Morava K-theory spectra~$\K(n)$ have
$\pi_*\K(n)\cong\bbF_p[v_n^{\pm1}]$ with an element~$v_n$ of degree
$2(p^n-1)$. For~$p$ odd, these are homotopy commutative ring spectra of characteristic~$p$. But, it is known that neither the spectra~$\K(n)$ nor their connective covers~$\moravak(n)$ admit an~$E_\infty$ structure. See Lemma 5.6.4~in \cite{rognes:galois}, for example, where this is proven for the~$\moravak(n)$ with the help of Dyer-Lashof operations. And, by~\cite{Baker+Richter:connective}, it suffices to prove the result for these. Of course, the non-existence of~$E_\infty$ multiplications also follows from Steinberger's Theorem~\ref{thm:steinberger:splitting}.


\subsection{Moore spectra}\label{sec:Moore}

To conclude this section, let me comment upon multiplications on the
Moore spectra~$\rmS^0/p$. These are the cofibers of the~$p$-multiplication on~$\rmS^0$. Additional information is contained in~\cite{araki+toda:1} and~\cite{araki+toda:2}.

At the prime~$2$, the Moore spectrum~$\rmS^0/2$ is not a commutative~$\bbS$-algebra. In fact, it is
not even a ring spectrum up to homotopy. One reason is
that~$\pi_2(\rmS^0/2)\cong\bbZ/4$ is not a module
over~$\pi_0(\rmS^0/2)\cong\bbZ/2$. Another reason is that there is
clearly no ring map from~$\pi_0(\rmS^0/2)\cong\bbZ/2$
to~\hbox{$\pi_0\mathrm{Map}(\rmS^0/2,\rmS^0/2)\cong\bbZ/4$}. And of course, it also follows from Theorem~\ref{thm:mod2isEM}, because~$\rmS^0/2$ is not an Eilenberg-Mac Lane spectrum, since multiplication by the Hopf map~$\eta$ is non-trivial on the homotopy groups. In fact, the element~$\eta$ itself would have to be zero, because it is the power operation on~$2=0$ in~$\pi_0(\rmS^0/2)$.
 
For other primes, one may use the power operation~$\beta
P^1\colon\pi_0(E)\rightarrow\pi_{2p-3}(E)$ for~commutative~$\bbS$-algebras~$E$ to
show that there is no~$E_\infty$ structure on~$\rmS^0/p$: One the one
hand~\hbox{$\beta P^1(0)=0$}, but on the other hand~$\beta P^1(p)$ is
a non-trivial multiple of~$\alpha_1$, see~\cite[V.1.13]{steinberger}, for example. This implies that the unit of a commutative~$\bbS$-algebra~$E$ of characteristic~$p$ must map~$\alpha_1$ to zero. But this is not the case for~$\rmS^0/p$.

The Moore spectrum~$\rmS^0/p$ is known to have an~$A_{p-1}$-multiplication which is not~$A_p$, see~\cite[Example 3.3]{Angeltveit}.


\section{Versal examples}\label{sec:SSp}

If~$A$ is an ordinary ring of characteristic~$p$, then every~$A$-algebra is also of characteristic~$p$, and~$\bbF_p$ is the initial object. However, as it turns out, there is no commutative~$\bbS$-algebra of characteristic~$p$ which is initial in the homotopical sense, see~Proposition~\ref{prop:not_initial}. The reason is that the existence of a homotopy~$p\simeq0$ does not imply the uniqueness of such a homotopy. Therefore, we will instead turn our attention to commutative~$\bbS$-algebras which come with a chosen homotopy~$p\simeq0$.

\subsection{The commutative~\texorpdfstring{$\bbS$}{S}-algebras~\texorpdfstring{$\bbSp$}{S//p}}

Let~$p$ be a prime number. As a motivation for the following definition, consider the description
of the corresponding prime field~$\bbF_p$ as a pushout
\begin{center}\mbox{\xymatrix{ 
\bbZ[T]\ar[r]^-{T\mapsto p}\ar[d]_-{T\mapsto 0} & \bbZ\ar[d]\\
\bbZ\ar[r] & \bbF_p
}}
\end{center}
in the category of rings. This can be built as a tensor product
\[
\bbF_p=\bbZ\otimes_{\bbZ[T]}\bbZ
\] 
using the structure maps indicated in the diagram. The tensor product is already the
derived tensor product: the higher~$\Tor$-terms vanish, since~$p$ is not a zero-divisor on~$\bbZ$. 

We will imitate this now in the category of~commutative~$\bbS$-algebras. Let us agree to write~$\calE_\infty(E,F)$ for the {\em derived} mapping space of~$E_\infty$ maps~$E\to F$, i.e.~we tacitly assume that~$E$ has been replaced by an equivalent commutative~$\bbS$-algebra that is cofibrant before computing the actual space of maps. (Similarly for~$\calS_\infty$ and spectra.) Let~$\bbP$ be the left adjoint to the forgetful functor. In other words,~$\bbP X$ is the free~commutative~$\bbS$-algebra on~$X$. If~$X$ is  cofibrant as a spectrum, then~$\bbP X$ is cofibrant as an~$E_\infty$ ring spectrum. There is an adjunction
\[
\calE_\infty(\bbP X,E)\cong\calS_\infty(X,E) 
\]
for commutative~$\bbS$-algebras~$E$, induced by the unit~$X\to\bbP X$ of the adjunction. The~$E_\infty$ map corresponding to~$x\colon X\to E$ will be denoted by~$\ev(x)\colon\bbP X\to E$; it is the~$E_\infty$ map which evaluates to~$x$ on the generator~$X$.

\begin{definition}\label{def:SSp}
The~commutative~$\bbS$-algebra~$\bbSp$ is the homotopy pushout
\begin{center}
  \mbox{
    \xymatrix{ 
      \bbP\rmS^0\ar[r]^{\ev(p)}\ar[d]_{\ev(0)} & \bbS\ar[d]\\
      \bbS\ar[r]& \bbSp
    }
  }
\end{center}
in~commutative~$\bbS$-algebras. 
\end{definition}

For the honest construction of the homotopy pushout we will have to replace the map~\hbox{$\ev(0)\colon\bbP\rmS^0\to\bbS=\bbP*$} by an equivalent cofibration, for example by~$\bbP(\rmS^0\to\rmC\rmS^0)$. In other words, the commutative~$\bbS$-algebra~$\bbSp$ is obtained from the sphere spectrum~$\bbS$ by attaching an~$E_\infty$-cell so as to ensure~$p\simeq0$. In particular, the commutative~$\bbS$-algebra~$\bbSp$ is cofibrant. It has the property, in analogy with the above description of~$\bbF_p$ as~\hbox{$\bbZ[T]/(0=T=p)$}, that an~$E_\infty$ map~$\bbSp\rightarrow E$ is the same as a null-homotopy from~$p$ to~$0$ in~$E$. In particular, there might be more than one such map, in contrast to the discrete case, where a ring map~$\bbF_p\rightarrow A$ is unique if it exists. We will discuss spaces of~$E_\infty$ maps out of~$\bbSp$ in more detail later, see Section~\ref{sec:Eoo_details}.

A~$p$-local commutative~$\bbS$-algebra~$Q$ is called {\it nuclear} if it is a colimit of commutative~$\bbS$-algebras~$Q_n$, where~$Q_0=\bbS$ and, inductively, the commutative~$\bbS$-algebra~$Q_{n+1}$ is obtained from the commutative~$\bbS$-algebra~$Q_n$ by coning off (in the~$E_\infty$ sense) finitely many elements in~$\pi_n(Q_n)$ in such a way that the kernel of the corresponding map~$\pi_n(\vee\rmS^n)\to\pi_n(Q_n)$ consists of multiples of~$p$. See~\cite[Definition~2.7]{Hu+Kriz+May}.

\begin{proposition}
For all primes~$p$, the commutative~$\bbS$-algebra~$\bbSp$ is nuclear.
\end{proposition}

\begin{proof}
Set~$\bbSp=Q=Q_1$. By Definition~\ref{def:SSp}, this is obtained from~$Q_0=\bbS$
by coning off (in the~$E_\infty$ sense) the element~$p\in\pi_0(Q_0)=\pi_0(\bbS)=\bbZ$. The assumption on the kernel is satisfied, because the corresponding map~\hbox{$\pi_0(\rmS^0)\to\pi_0(\bbS)$} that sends the identity to~$p$ is even injective.
\end{proof}

The following observation will be useful later, in Section~\ref{sec:Thom}, when we relate the spectra~$\bbSp$ to Thom spectra.

\begin{proposition}\label{prop:iterated}
The~commutative~$\bbS$-algebra~$\bbSp$
may be described as the following iterated homotopy pushout in~commutative~$\bbS$-algebras.
\begin{center}
\mbox{
\xymatrix{ 
\bbP\rmS^0\ar[r]^{\bbP(p)}\ar[d]_{\ev(0)} & \bbP\rmS^0\ar[d]\ar[r]^{\ev(1)} & \bbS\ar[d]\\
\bbS\ar[r] & \bbP(\rmS^0/p)\ar[r] & \bbSp
}
}
\end{center}
\end{proposition}

This statement is analogous to the description of~$\bbF_p$
as~$\bbZ[T]/(p=0,T=1)$. Note that the left hand square is a pushout because
the left adjoint~$\bbP$ commutes with colimits.

\begin{proof}
  This follows immediately from the fact that the 
  map~$\ev(p)\colon\bbP\rmS^0\rightarrow \bbS$ factors as the composition of~$\bbP(p)$
  and~\hbox{$1\colon\bbP\rmS^0\rightarrow\bbS$}, as can be seen by
  composition with the unit~\hbox{$\bbS\rightarrow\bbP\rmS^0$}.
\end{proof}


\subsection{The underlying ring of components}

At this point, one might wish to compute the homology and homotopy
of~$\bbSp$ directly or with the help of the Hurewicz homomorphism. This will be
done later, in Section~\ref{sec:homology_and_homotopy}, see in particular Proposition~\ref{prop:homotopy}. For a start, only the following result will be needed.

\begin{proposition}\label{prop:pi0}
We have~$\pi_0(\bbSp)=\bbF_p$.
\end{proposition}

\begin{proof}
  Because the spectrum~$\bbSp$ is connective, it suffices to compute the integral
  homology of it. Smashing the defining homotopy pushout diagram for~$\bbSp$ with
  the integral Eilenberg-Mac Lane spectrum yields another homotopy pushout diagram, and
  the integral Eilenberg-Moore spectral sequence then implies
  that~$\pi_0$ must be~\hbox{$\bbZ[T]/(0=T=p)=\bbF_p$}, as in the
  motivation given for the definition of~$\bbSp$ at the beginning.
\end{proof}

This result implies that we have
\[
\ch(\bbSp)=p,
\]
as it should be.


\subsection{Spaces of~\texorpdfstring{$E_\infty$}{Einfty} maps out of~\texorpdfstring{$\bbSp$}{S//p}}\label{sec:Eoo_details}

The following result identifies the space of~$E_\infty$ maps~$\bbSp\rightarrow E$ into any commutative~$\bbS$-algebra~$E$.

\begin{proposition}\label{prop:maps_to_E}
  Let~$E$ be a~commutative~$\bbS$-algebra. If~$\ch(E)=p$, then there is an equivalence
\[
\calE_\infty(\bbSp,E)\simeq\Omega^{\infty+1}(E)
\]
of spaces. Otherwise, the left hand side is empty.
\end{proposition}

\begin{proof}
  The definition of~$\bbSp$ as a homotopy pushout implies that there is a homotopy pullback
  \begin{center}
    \mbox{
      \xymatrix{ 
        \calE_\infty(\bbSp,E)\ar[d]\ar[r] &
        \calE_\infty(\bbS,E)\ar[d]\\
        \calE_\infty(\bbS,E) \ar[r]&
        \calE_\infty(\bbP\rmS^0,E).
      }
    }
  \end{center}
  Now use that the spaces~$\calE_\infty(\bbS,E)$ are contractible, and that
  their images in~$\calE_\infty(\bbP\rmS^0,E)$ can be connected by a path if and only if~$p\simeq0$. The result then follows using the adjunction~$\calE_\infty(\bbP\rmS^0,E)\cong\calS_\infty(\rmS^0,E)$
  and the definition~$\calS_\infty(\rmS^0,E)=\Omega^\infty(E)$.
\end{proof}

\begin{example}
If~$E=\rmH A$ for some discrete commutative ring~$A$ of characteristic~$p$, then the space~$\calE_\infty(\bbSp,\rmH A)\simeq\Omega^{\infty+1}\rmH A$ is contractible. The following result states that the converse also holds in the connective case.
\end{example}

\begin{proposition}
If~$E$ is a connective commutative~$\bbS$-algebra, with the property that the space~$\calE_\infty(\bbSp,E)$ is contractible, then~$E\simeq\rmH\pi_0(E)$ is discrete.
\end{proposition}

\begin{proof}
If~$\Omega^{\infty+1}(E)$ is contractible, then the higher homotopy groups of~$E$ vanish, so that~$E$ is an Eilenberg-Mac Lane spectrum.
\end{proof}

\begin{example}\label{ex:Eoo_maps_SSt_MO}
Let us consider the unoriented bordism spectrum~$\MO$. Since the underlying ring is~$\pi_0(\MO)=\bbF_2$, this spectrum satisfies
\[
\ch(\MO)=2.
\]
Because~$\pi_1(\MO)$ is trivial, it admits an~$E_\infty$ map from~$\bbSt$ which is unique up to homotopy (of~$E_\infty$ maps). However, since~$\pi_2(\MO)\not=0$, the space of all such maps is not contractible.
\end{example}

\begin{corollary}\label{cor:not_contractible}
There is an equivalence~$\calE_\infty(\bbSp,\bbSp)\simeq\Omega^{\infty+1}(\bbSp)$.
\end{corollary}

We will see later, in Proposition~\ref{prop:homotopy}, that~$\pi_1(\bbSp)=0$, but that~$\bbSp$ has non-trivial higher homotopy groups, so that the space~$\calE_\infty(\bbSp,\bbSp)$ is connected but not contractible.


\subsection{Commutative~\texorpdfstring{$\bbSp$}{S//p}-algebras}

Let us begin by recalling the following definition.

\begin{definition}
A {\it commutative~$\bbSp$-algebra} is a commutative~$\bbS$-algebra~$E$ together with an~$E_\infty$ map~$s_E\colon\bbSp\to E$, the {\it structure map}. The space~$\calE_\infty(\bbSp,E)$ is the {\it space of~$\bbSp$-algebra structures on~$E$}.
\end{definition}

Proposition~\ref{prop:maps_to_E} and the examples which follow it show that being a commutative~$\bbSp$-algebra is not a {\it property} of~commutative~$\bbS$-algebras, but an extra {\it structure}, which
if it exists, need not be unique. And if it is unique, it need not be canonically so. In more conceptual terms, the property~$\ch(E)=p$ defines a full subcategory of the category of commutative~$\bbS$-algebras. Rather than work in this full subcategory defined by a property, it seems better to keep track of the choice of a structure map~$s_E\colon\bbSp\to E$, and work in the category~$\calE_\infty^{\bbSp}$ of commutative~$\bbSp$-algebras.

If~$E$ and~$F$ are commutative~$\bbSp$-algebras, then there is a homotopy fibration square
\begin{center}
  \mbox{ 
    \xymatrix{
\calE_\infty^{\bbSp}(E,F)\ar[r]^-\subseteq\ar[d] & \calE_\infty(E,F)\ar[d]^{s_E^*}\\
\{s_F\}\ar[r]_-\subseteq& \calE_\infty(\bbSp,F)
    } 
  }
\end{center}
for the derived mapping spaces. Therefore, Proposition~\ref{prop:maps_to_E} also describes the difference between the space of all~$E_\infty$ maps and that of~$\bbSp$-algebra maps. 

For example, in contrast to Corollary~\ref{cor:not_contractible}, it is clear that there is an essentially unique map~$\bbSp\rightarrow\bbSp$ of algebras over~$\bbSp$. Compare with Proposition~\ref{prop:not_initial}.


\section{Homology and homotopy of the versal examples}\label{sec:homology_and_homotopy}

In this section, we will give some basic information on the homotopy type of~$\bbSp$ and
the associated algebraic invariants: its homology and homotopy groups. 

\begin{proposition}
For each prime~$p$, the spectrum~$\bbSp$ is a graded Eilenberg-Mac Lane spectrum.
\end{proposition}

\begin{proof}
Since~$\pi_0(\bbSp)=\bbF_p$ by Proposition~\ref{prop:pi0}, this is a consequence of Steinberger's result, see Theorem~\ref{thm:steinberger:splitting}.
\end{proof}

As a consequence, the homotopy type of~$\bbSp$ can be read off from its homotopy groups, and the homotopy groups of~$\bbSp$ will be computable once the~(mod~$p$) homology is known. 


\subsection{The homology of~\texorpdfstring{$\bbSp$}{S//p}}

The backbone of the homology computation is the following result.

\begin{proposition}\label{prop:A-algebra_equivalence}
If~$A$ is a commutative~$\bbS$-algebra of characteristic~$p$, then there exists an equivalence
\[
A\wedge\bbSp\simeq A\wedge\bbP\rmS^1
\]
of commutative~$A$-algebras.
\end{proposition}

\begin{proof}
Both~$A$-algebras are homotopy pushouts of diagrams
  \begin{center}
    \mbox{ 
      \xymatrix@C=50pt{
        A\wedge\bbP\rmS^0\ar[r]^-{A\wedge\ev(0)}\ar[d]& A\wedge\bbS\\
        A\wedge\bbS,&
      } 
    }
  \end{center}
  where the left arrow is~$A\wedge\ev(p)$ and~$A\wedge\ev(0)$, respectively. The space of 
  maps of~$A$-algebras from~$A\wedge\bbP \rmS^0$ to~$A\wedge\bbS\simeq A$ is, by the adjunctions
  \[
\calE_\infty^{A}(A\wedge\bbP\rmS^0,A)\cong
\calE_\infty(\bbP \rmS^0,A)\cong
\calS_\infty(\rmS^0,A)=\Omega^\infty A,
\]
   equivalent to the underlying infinite loop spaces of~$A$. (Note that~\hbox{$A\wedge\bbP\rmS^0$} is cofibrant as a commutative~$A$-algebra.) Since~$p$ and~$0$ are in the same component by hypothesis on~$A$, these maps of commutative~$A$-algebras are homotopic. As a result, the homotopy pushouts are equivalent.
\end{proof}

\begin{corollary}\label{cor:A-algebra_equivalence}
If~$A$ is a commutative~$\bbS$-algebra with~$\ch(A)=p$, then there is an isomorphism
\[
A_*(\bbSp)\cong A_*(\bbP\rmS^1)
\]
of~$A_*$-algebras. 
\end{corollary}

We hasten to point out that neither the equivalence in Proposition~\ref{prop:A-algebra_equivalence}, nor the isomorphism in Corollary~\ref{cor:A-algebra_equivalence} can be induced by a map between~$\bbSp$ and~$\bbP\rmS^1$. If it were, then the isomorphism would be compatible with the natural~$A$-homology operations for all~$A$. Example~\ref{ex:homology_Bockstein} below shows that this need not be the case.

The preceding results can be applied in the case~$A=\bbSp$ itself, and in the case~$A=\MO$ when~$p=2$. We will mostly be interested in the case~$A=\rmH\bbF_p$, when the preceding corollary shows the existence of an isomorphism
\begin{equation}\label{eq:H_iso}
\rmH_*(\bbSp)\cong\rmH_*(\bbP\rmS^1).
\end{equation}
The right hand side is the free algebra over the Araki-Kudo-Dyer-Lashof algebra on one generator in degree~$1$, see~\cite[II.4]{cohen+lada+may} and~\cite{Baker:free}. In particular, 
\begin{equation}\label{eq:H_1}
\rmH_1(\bbSp)\cong\rmH_1(\bbP\rmS^1)\cong\bbZ/p.
\end{equation}
The consequences for the uniqueness of the isomorphism~\eqref{eq:H_iso} will be discussed below.

\begin{example}\label{ex:homology_Bockstein}
The natural~$A$-homology operations on~$A_*(\bbSp)$\ and~$A_*(\bbP\rmS^1)$
differ in the case~$A=\rmH\bbF_p$ of the Eilenberg-Mac Lane spectrum. To see this, note that on the one hand we have an isomorphism~$\rmH\bbZ_0(\bbP\rmS^1)=\bbZ$, while on the other hand we have an isomorphism~$\rmH\bbZ_0(\bbSp)=\bbF_p$. Consequently, the homology Bockstein from dimension~$1$ to~$0$ is zero on the homology of~$\bbP\rmS^1$, but an isomorphism on the homology of~$\bbSp$.
\end{example}

In the case~$p=2$, the homology of the cell commutative~$\bbS$-algebra~$\bbSt$ can also be obtained as a special case of Baker's calculations in~\cite{Baker:BP}.

Let us now address the question of the uniqueness of the equivalence in Proposition~\ref{prop:A-algebra_equivalence} in the specific case~$A=\rmH\bbF_p$.

\begin{proposition}
There exist precisely~$p-1$ homotopy classes of equivalences
\[
\rmH\bbF_p\wedge\bbP\rmS^1\longrightarrow\rmH\bbF_p\wedge\bbSp
\]
of commutative~$\rmH\bbF_p$-algebras.
\end{proposition}

\begin{proof}
There are adjunctions
\[
\calE_\infty^{A}(A\wedge\bbP\rmS^1,A\wedge\bbSp)\cong
\calE_\infty(\bbP\rmS^1,A\wedge\bbSp)\cong
\calS_\infty(\rmS^1,A\wedge\bbSp),
\]
so that the homotopy classes of maps~$A\wedge\bbP\rmS^1\to A\wedge\bbSp$ of commutative~$A$-algebras are parametrized by
\[
\pi_1(A\wedge\bbSp)=A_1(\bbSp).
\]
In the present case~$A=\rmH\bbF_p$ this means that there are precisely~$p$ homotopy classes of maps 
$\rmH\bbF_p\wedge\bbP\rmS^1\to\rmH\bbF_p\wedge\bbSp$ of commutative~$\rmH\bbF_p$-algebras, and one of them (the zero) is not an equivalence.
\end{proof}

Clearly, for~$p=2$, we have~$p-1=1$, so that there is in fact a unique equivalence
\[
\rmH\bbF_2\wedge\bbP\rmS^1\simeq\rmH\bbF_2\wedge\bbSt
\]
of~$\rmH\bbF_2$-algebras! The non-uniqueness at the odd primes~$p$ comes from the~$\bbF_p^\times$-action on the~$1$-dimensional~$\bbF_p$-vector space~$\rmH_1(\rmS^1)$. It seems fair to say that this is well under control.


\subsection{Applications to lifts of~\texorpdfstring{$\rmH\bbZ/p$}{HZ/p}}

It may be worthwhile to point out that the same process that produces~$\bbSp$ from~$\bbS$ does not lead to~$\rmH\bbZ/p$ when applied to~$\rmH\bbZ$. More precisely, if we kill~$p$ in~$\bbZ=\pi_0\rmH\bbZ$ to form the~$E_\infty$ ring spectrum
\[
\rmH\bbZ\modmod p\simeq\rmH\bbZ\wedge\bbSp
\]
then there is an~$E_\infty$ map
\[
\rmH\bbZ\modmod p\longrightarrow\rmH\bbZ/p,
\]
but this is not an equivalence:

\begin{proposition}
The spectra~$\rmH\bbZ\modmod p$ and~$\rmH\bbZ/p$ are not equivalent.
\end{proposition}

\begin{proof}
Let us apply the functor~$\rmH\bbZ/p\underset{\rmH\bbZ}{\wedge}?$ to both sides. By Proposition~\ref{prop:A-algebra_equivalence}, we get
\[
\rmH\bbZ/p\underset{\rmH\bbZ}{\wedge}(\rmH\bbZ\wedge\bbSp)\simeq\rmH\bbZ/p\wedge\bbSp\simeq\rmH\bbZ/p\wedge\bbP\rmS^1
\]
on the left hand side, and the spectrum~$\rmH\bbZ/p\underset{\rmH\bbZ}{\wedge}\rmH\bbZ/p$ with
\[
\pi_*(\rmH\bbZ/p\underset{\rmH\bbZ}{\wedge}\rmH\bbZ/p)=\Tor^{\bbZ}_*(\bbZ/p,\bbZ/p)
\]
on the right, and these are clearly different.
\end{proof}

It follows that the unit~$\bbS\to\bbSp$ is {\em not} a lift of the `extension'~$\bbZ\to\bbZ/p$ in the sense of~\cite[Definition~1]{SVW}. In fact, such a lift cannot exist: If~$R$ were a connective~$E_\infty$ ring spectrum such that~$\rmH\bbZ\wedge R\simeq\rmH\bbZ/p$, then~$R$ would have to be a Moore spectrum, in contradiction to the results in Section~\ref{sec:Moore}.


\subsection{Some homotopy groups of~\texorpdfstring{$\bbSp$}{S//p}}

The following is our device for passing from homology to homotopy.

\begin{proposition}\label{prop:poincare}
The Hurewicz map
$\pi_*(\bbSp)\to\rmH_*(\bbSp)$
is injective, and there is an isomorphism
$\rmH_*(\bbSp)\cong\pi_*(\bbSp)\otimes\rmH_*(\rmH\bbF_p)$.
\end{proposition}

\begin{proof}
This follows immediately from the fact that~$\bbSp$ is a graded Eilenberg-Mac Lane spectrum.
\end{proof}


As a consequence of Proposition~\ref{prop:poincare}, the Poincar\'e series of the homotopy of~$\bbSp$ is the quotient of the Poincar\'e series of the homology of~$\bbSp$ by the Poincar\'e series of the dual Steenrod algebra. The Poincar\'e series of the (dual) Steenrod algebra is well-known, and the Poincar\'e series of~$\rmH_*(\bbSp)\cong\rmH_*(\bbP\rmS^1)$ is a matter of combinatorics, because this algebra is free on admissible generators of prescribed excess. However, a closed formula does not seem to be in the literature, and we will not pursue this here, either. For our purposes, it will be sufficient to determine the  homotopy groups in low dimensions.

\begin{proposition}\label{prop:homotopy}
In low dimensions, the homotopy groups of~$\bbSp$ are as follows.
\[
\pi_n(\bbSp)\cong
\begin{cases}
\bbZ/p & \text{if }n=0\\
0 & \text{if }0<n<4(p-1)\\
\bbZ/p& \text{if }n=4(p-1)
\end{cases}
\]
\end{proposition}

\begin{proof}
  Let me explain this for odd primes~$p$. The case~$p=2$ is
  similar but easier; see the proof of Proposition~\ref{prop:DyerLashof} for hints.
  
  In degrees at most~$4(p-1)$, the additive generators of the (dual)
  Steenrod algebra are
  \[
    1,\tau_0,\xi_1,\tau_0\xi_1,\tau_1,\tau_0\tau_1, \xi_1^2.
  \]
  In the free algebra over the Dyer-Lashof algebra on one
  generator~$a$ in dimension~$1$, the corresponding generators are 
  \[
    1, a, \beta Q^1a, a\beta Q^1a, Q^1a, aQ^1a, (\beta Q^1a)^2.
  \]
  But, these are not the only classes in degrees at
  most~$4(p-1)$: in degree~$4(p-1)$ itself, there is not only~$(\beta
  Q^1a)^2$ but also~$\beta Q^2a$, and these two are linearly
  independent.
\end{proof}


\subsection{An application to~\texorpdfstring{$E_\infty$}{Einfty} self-maps of~\texorpdfstring{$\bbSp$}{S//p}}

The equivalence~$\calE_\infty(\bbSp,\bbSp)\simeq\Omega^{\infty+1}(\bbSp)$ from Corollary~\ref{cor:not_contractible} implies that there are isomorphisms~$\pi_n\calE_\infty(\bbSp,\bbSp)\cong\pi_{n+1}(\bbSp)$ for all~$n\geqslant0$. In particular, we have~$\pi_0\calE_\infty(\bbSp,\bbSp)=\pi_1(\bbSp)=0$ by the preceding Proposition~\ref{prop:homotopy}, which also yields the non-triviality of some higher homotopy groups.

\begin{corollary}\label{cor:id_S}
  Every~$E_\infty$ self-map of~$\bbSp$ is~$E_\infty$ homotopic to the
  identity. But, the space of~$E_\infty$ self-map of~$\bbSp$ is not contractible: The first non-trivial homotopy group is in dimension~$4p-5$.
\end{corollary}

In particular, there is an essential~$3$-sphere in the space of~$E_\infty$ self-map of~$\bbSt$.

In~\cite[Definition~2.8]{Hu+Kriz+May}, a connective commutative~$\bbS$-algebra~$E$ whose unit~\hbox{$\bbS\to E$} induces an isomorphism on underlying rings is called {\it atomic} if every self-map of~$\bbS$-algebras~\hbox{$E\to E$} is a weak equivalence. This property is already useful in the case when the unit is only surjective, and the following statement uses that terminology in this broader sense.

\begin{corollary}
For each prime~$p$, the commutative~$\bbS$-algebra~$\bbSp$ is atomic.
\end{corollary}


\subsection{Initial commutative~\texorpdfstring{$\bbS$}{S}-algebras of prime characteristic}

These do not exist! While the ordinary commutative ring~$\bbF_p$ is initial among commutative rings of characteristic~$p$, there is no commutative~$\bbS$-algebra of characteristic~$p$ which is initial in the homotopical sense. This is the content of the following result.

\begin{proposition}\label{prop:not_initial}
If~$p$ is a prime number, then there is no commutative~$\bbS$-algebra~$T$ with~$\ch(T)=p$ such that the derived space of commutative~$\bbS$-algebra maps~$T\to E$ is contractible for all commutative~$\bbS$-algebras~$A$ with~$\ch(E)=p$.
\end{proposition}

\begin{proof}
Assume that there were such a commutative~$\bbS$-algebra~$T$. Because we have assumed~$\ch(T)=p$, there is a structure map~$s\colon\bbSp\to T$. We will first prove that this is an equivalence. By hypothesis, the mapping spaces
\begin{equation}\label{eq:t}
\calE_\infty(T,\bbSp)\simeq*
\end{equation}
and
\begin{equation}\label{eq:id_T}
\calE_\infty(T,T)\simeq*
\end{equation}
are contractible. We will use \eqref{eq:t} to pick a map~$t\colon T\to\bbSp$ of commutative~$\bbS$-algebras. This is an inverse (up to homotopy) of~$s$: By \eqref{eq:id_T} and Corollary~\ref{cor:id_S}, both compositions~$st$ and~$ts$ are homotopic to the identities via maps of commutative~$\bbS$-algebras. This implies that~($T^\cof$ and then)~$T$ is equivalent to~$\bbSp$. 

Thus, if~$T$ exists, then~$T\simeq\bbSp$. But, we have already seen that~$\bbSp$ does not satisfy the strong uniqueness as in the statement of the proposition. For example, the derived space of commutative~$\bbS$-algebra maps of~$\bbSp\to E$ is not contractible for~\hbox{$E=\bbSp$} itself, by Corollary~\ref{cor:id_S} above.
\end{proof}


\section{The versal examples as Thom spectra}\label{sec:Thom}

The aim of this section is to identify the spectra~$\bbSp$ for the various primes
with certain~$E_\infty$ Thom spectra. This is clearly useful, as it
will allow us to relate the spectra~$\bbSp$ to other Thom spectra, such as the unoriented bordism spectrum~$\MO$ if~$p=2$, and it will also allow for the description of the Hochschild and Andr\'e-Quillen invariants of the versal examples. 

\subsection{Thom spectra}

There is a Thom spectrum~$\M_f$ associated with every stable spherical fibration, classified by a map~$f\colon X\rightarrow \rmB\GL_1(\bbS)$, on a connected space~$X$. In order to deal with the~$\bbSp$ for odd primes~$p$, we also require the generalized Thom spectra of~\cite{Blumberg}, and~\cite{ABGHR}, where~$\bbS$ can be replaced by the~$p$-local or~$p$-adic sphere.

\begin{example}\label{ex:MO_as_Thom}
The original and most prominent example of a Thom spectrum is certainly given by the embedding
$\BO\to\rmB\GL_1(\bbS)$ of the linear isomorphisms into the homotopy equivalences: This gives rise to the spectrum~$\MO$ for unoriented bordism.
\end{example}

\begin{example}\label{ex:sphere_as_Thom}
If~$X$ is a point, then the Thom spectrum~$\M_f$ is equivalent to the sphere spectrum~$\bbS$. More generally, if~$f$ is null-homotopic, so that it classifies the trivial bundle, then the Thom spectrum is equivalent to the suspension spectrum~$\bbS\wedge X_+$. 
\end{example}

\begin{example}\label{ex:Moebius_as_Thom}
The generator of~$\pi_1\rmB\GL_1(\bbS)=\bbZ/2$ classifies the (stable) M\"obius bundle,
and the associated Thom spectrum is the Moore spectrum~$\rmS^0/2$ at the prime~$p=2$.
\end{example}

If~$X$ is an infinite loop space and~$f$ is an infinite loop map, then the Thom spectrum~$\M_f$ is an~$E_\infty$ ring spectrum. See~\cite[IX.7]{lewis}. This applies in Example~\ref{ex:MO_as_Thom}
as well as in Example~\ref{ex:sphere_as_Thom}. But, if~$X$ is only a two-fold loop space and~$f$ is a two-fold loop map, then we can only infer that~$\M_f$ is a commutative ring spectrum up to homotopy.

\begin{example}\label{ex:Mahowald}
We may apply this to the two-fold delooping~\hbox{$\Omega^2 \rmS^3\rightarrow\BO$} of the
previous Example~\ref{ex:Moebius_as_Thom}. The resulting Thom spectrum is known to be the
Eilenberg-Mac Lane spectrum~$\rmH\bbF_2$. See~\cite{mahowald}. Note that this
turns out to admit an~$E_\infty$ multiplication, but this is not clear from its
construction as a Thom spectrum. 
\end{example}


\subsection{The examples~\texorpdfstring{$\bbSt$}{S//2} and~\texorpdfstring{$\bbSp$}{S//p}}

We are now able to show that the examples~$\bbSt$ and~$\bbSp$ can be realized as Thom spectra.

\begin{theorem}\label{thm:Thom_2}
The spectrum~$\bbSt$ is the~$E_\infty$ Thom spectrum of the infinite delooping
\[
\rmQ(\rmS^1)\to\rmB\GL_1(\bbS)
\]
of the classifying map of the M\"obius bundle.
\end{theorem}

\begin{proof}
  The free~$E_\infty$-spectrum on the Thom spectrum of~$f$ is the Thom
  spectrum of~$\rmQ(f)$, see~\cite{lewis}, Theorem 7.1 on page
  444. Beware that the base cell acts as a unit. Therefore, the identification of the Thom spectrum of the essential map~$\rmS^1\to\BO$ with the Moore spectrum, and the
  description of~$\bbSt$ as an iterated pushout from
  Proposition~\ref{prop:iterated} immediately imply the result.
\end{proof}

An argument due to Hopkins, see~\cite{Mahowald+Ravenel+Shick} or~\cite{Blumberg}, allows us to extend the preceding result to odd primes. We need to know that, also for these primes~$p$, the Moore spectrum~$\rmS^0/p$ is the Thom spectrum of a map
\begin{equation}\label{eq:hopkins map}
f\colon\rmS^1\longrightarrow\rmB\GL_1(\bbS_p),
\end{equation}
where the target now classifies stable~$p$-adic spherical fibrations, and such an~$f$ is a class in~\hbox{$\pi_1(\rmB\GL_1(\bbS_p))\cong\bbZ_p^\times$}, the group of~$p$-adic units. To obtain the Moore spectrum we may chose~$f$ to be a representative of the unit~$1-p$. Now a similar argument as above implies that~$\bbSp$ can be obtained as the Thom
spectrum of the infinite delooping of the map~\eqref{eq:hopkins map}.

\begin{theorem}\label{thm:Thom_odd}
The spectrum~$\bbSp$ is the~$E_\infty$ Thom spectrum of the infinite delooping
\[
\rmQ(\rmS^1)\to\rmB\GL_1(\bbS_p)
\]
of the map~$f$ such that~$\M_f$ is the Moore spectrum.
\end{theorem}


\subsection{The topological Hochschild homology of~\texorpdfstring{$\bbSp$}{S//p}}

Recall that the topological Hochschild homology spectrum~$\THH^{\bbS}(E)$ of a commutative~$\bbS$-algebra~$E$ is an important invariant of~$E$, not the least because it is an approximation to the algebraic K-theory of~$E$. We will now determine it for~$\bbSt$ and~$\bbSp$, based on the identification of these spectra as~$E_\infty$ Thom spectra, and general results due to Blumberg~\cite{Blumberg} which apply for this class.

\begin{theorem}\label{thm:THH}
For each prime number~$p$ there is an equivalence 
\[
\THH^{\bbS}(\bbSp)\simeq\bbSp\wedge\rmQ(\rmS^2)_+
\]
of spectra.
\end{theorem}

\begin{proof}
The topological Hochschild homology of commutative~$\bbS$-algebras which are Thom spectra has been determined by Blumberg, see~\cite[Theorem 1.5]{Blumberg}. He shows that, if~$f\colon X\to\rmB\GL_1(\bbS)$ 
is a map of~$E_\infty$ spaces which is good (a fibration, for example) and such that~$X$ is a cofibrant and group-like~$E_\infty$-space, then there is a weak equivalence of commutative~$\bbS$-algebras as follows.
\[
\THH^{\bbS}(\M_f)\simeq\M_f\wedge\rmB X_+
\]
If~$X$ is only a group like~$E_2$-space, and~$f$ is only a good map of~$E_2$-spaces, then Blumberg can still show that there is an equivalence of spectra as above, provided that a least the homotopy commutative multiplication on~$\M_f$ admits an~$E_\infty$ refinement, see~\cite[Theorem 1.6]{Blumberg}.

The spectra~$\bbSt$ and~$\bbSp$ in question have been identified as Thom spectra in Theorems~\ref{thm:Thom_2} and~\ref{thm:Thom_odd}. Once we have replaced the relevant map~$f$ by a fibration, we may apply this theory, which works the same if~$\bbS$ is replaced by~$\bbS_p$. We obtain an equivalence
\[
\THH^{\bbS}(\bbSp)\simeq\bbSp\wedge\rmB\rmQ(\rmS^1)_+
\]
of spectra, and it remains to note that~$\rmB\rmQ(\rmS^1)\simeq\rmQ(\rmS^2)$.
\end{proof}

An extra argument is needed to obtain an equivalence of commutative~$\bbS$-algebras in Theorem~\ref{thm:THH}: One has to ensure that the fibrant replacement is still sufficiently well-behaved with respect to the smash product. However, this will not be used in the following.

To round off the discussion of topological Hochschild homology, let us also remind ourselves that the Thom spectrum of the canonical map~$\BO\to\rmB\GL_1(\bbS)$ is~$\MO$, so that Blumberg obtains equivalences
\begin{equation}\label{eq:thh_MO}
\THH^{\bbS}(\MO)\simeq\MO\wedge\rmB\BO_+\simeq\MO\wedge\widetilde{\UmodO}_+
\end{equation}
of commutative~$\bbS$-algebras, where~$\widetilde{\UmodO}$ is the universal cover of~$\UmodO$. He also obtains equivalences
\begin{equation}\label{eq:thh_HF}
\THH^{\bbS}(\rmH\bbF_p)\simeq\rmH\bbF_p\wedge\Omega\rmS^3_+
\end{equation}
of spectra, which shed new light on B\"okstedt's calculation of~$\THH^{\bbS}_*(\rmH\bbF_p)$.


\subsection{The cotangent complex of~\texorpdfstring{$\bbSp$}{S//p}}

We will now determine the topological Andr\'e-Quillen invariants of~$\bbSt$ and~$\bbSp$. As with our calculation of the topological Hochschild invariants, this can be based on the identification of these spectra as~$E_\infty$ Thom spectra, and general results which apply for this class, this time due to Basterra and Mandell~\cite{Basterra+Mandell}. However, we will also present a more direct approach which leads to the same result.

Recall that the Andr\'e-Quillen invariants of an extension~$F/E$ are defined using the~$F$-module contangent complex~$\Omega_E(F)$ which classifies derivations. The result for~$\bbSt$ and~$\bbSp$ (over~$\bbS$) is as follows.

\begin{theorem}\label{thm:TAQ_SSp}
For each prime number~$p$ there is an equivalence
\[
\Omega_{\bbS}(\bbSp)\simeq\Sigma\bbSp
\]
of~$\bbSp$-modules.
\end{theorem}

\begin{proof}[Proof of Theorem~\ref{thm:TAQ_SSp} for~$p=2$ using Thom technology]
In general, if~$T$ is a connective spectrum, and if~$f\colon\Omega^\infty(T)\to\rmB\GL_1(\bbS)$ is an~$E_\infty$ map, then a corollary of~\cite[Theorem~5]{Basterra+Mandell} identifies the cotangent complex of the~$E_\infty$ Thom spectrum~$\M_f$ of~$f$: There is an equivalence
\[
\Omega_{\bbS}(\M_f)\simeq\M_f\wedge T
\]
of~$\bbS$-modules. We may apply this theory to the spectrum~$\bbSt$, because is has been identified as a Thom spectrum of this type in Theorem~\ref{thm:Thom_2}.
 In this case, we may take~\hbox{$T=\rmS^1$} so that we obtain~$\Omega^\infty(T)=\rmQ(\rmS^1)$, and the result follows.
\end{proof}

The preceding proof would immediately generalize to the case of odd primes~$p$ as soon as the work of Basterra and Mandell would be extended to cover Thom spectra for~$p$-adic spherical fibrations as in Blumberg's work. For the time being, we will here provide for another proof which uses more traditional techniques associated with cotangent complexes.

\begin{proof}[Proof of Theorem~\ref{thm:TAQ_SSp} for all primes~$p$ using base change and transitivity]
Since~$\bbSp$ is defined as a homotopy pushout (Definition~\ref{def:SSp}), the flat base change formula~\cite[Proposition~4.6]{Basterra} applies to give an equivalence
\[
\Omega_{\bbS}(\bbSp)\simeq\Omega_{\bbP\rmS^0}(\bbS)\wedge_{\bbP\rmS^0}\bbS
\]
of~$\bbSp$-modules. Since it will be important to keep track of the maps~$\bbP\rmS^0\to\bbS$ of~$E_\infty$ algebras involved, let us agree that we use~$\ev(0)$ in~$\Omega_{\bbP\rmS^0}(\bbS)$ and~$\ev(p)$ in the base change~\hbox{$?\wedge_{\bbP\rmS^0}\bbS$}.

In order to determine~$\Omega_{\bbP\rmS^0}(\bbS)$, we may invoke the transitivity exact sequence~\cite[Proposition~4.2]{Basterra}. For the extensions~$\bbS\to\bbP\rmS^0\to\bbS$, it 
yields a fibration sequence
\[
\Omega_{\bbS}(\bbP\rmS^0)\wedge_{\bbP\rmS^0}\bbS
\longrightarrow
\Omega_{\bbS}(\bbS)
\longrightarrow
\Omega_{\bbP\rmS^0}(\bbS)
\]
in spectra. Now the middle term~$\Omega_{\bbS}(\bbS)$ is contractible. In general, the cotangent complex~$\Omega_{\bbS}(\bbP T)\simeq\bbP T\wedge T$ of the free commutative~$\bbS$-algebra on~$T$ is the free~$\bbP T$-module on~$T$, see~\cite[Example~3.8]{Kuhn},~\cite[Proposition 1.6]{Baker+Gilmour+Reinhard}, and the Appendix to~\cite{Baker:free}. In particular, there is an equivalence~$\Omega_{\bbS}(\bbP\rmS^0)\simeq\bbP\rmS^0\wedge\rmS^0\simeq\bbP\rmS^0$. This shows
\[
\Omega_{\bbP\rmS^0}(\bbS)\simeq\Sigma(\bbP\rmS^0\wedge_{\bbP\rmS^0}\bbS)\simeq\Sigma\bbS
\]
as~$\bbP\rmS^0$-modules, so that
\[
\Omega_{\bbP\rmS^0}(\bbS)\wedge_{\bbP\rmS^0}\bbS\simeq\Sigma(\bbS\wedge_{\bbP\rmS^0}\bbS)\simeq\Sigma\bbSp
\]
as~$\bbSp$-modules. As often before, we have used that the forgetful functor from commutative~$\bbS$-algebras to spectra is a right adjoint, so that it commutes with limits. Here, this determines the homotopy type of the pushouts involved.
\end{proof}

In~\cite[Definition 3.1]{Baker+Gilmour+Reinhard}, a~$p$-local commutative~$\bbS$-algebra with a CW structure is called {\it minimal} if for each~$n$ the inclusion of the~$n$-skeleton induces an isomorphism in topological Andr\'e-Quillen homology~$\TAQ^{\bbS}_n(?;\bbF_p)$.

\begin{corollary}
For each prime~$p$ the commutative~$\bbS$-algebra~$\bbSp$ is minimal.
\end{corollary}

\begin{proof}
By Theorem~\ref{thm:TAQ_SSp}, we have
\[
\TAQ^{\bbS}(\bbSp;\bbF_p)
=\Omega_{\bbS}(\bbSp)\wedge_{\bbSp}\rmH\bbF_p
\simeq\Sigma\bbSp\wedge_{\bbSp}\rmH\bbF_p
\simeq\Sigma\rmH\bbF_p,
\]
so that
\[
\TAQ^{\bbS}_n(\bbSp;\bbF_p)=\pi_n\TAQ^{\bbS}(\bbSp;\bbF_p)\cong
\begin{cases}
\bbZ/p& n=1\\
0& n\not=1.
\end{cases}
\]
Since~$\bbSp$, by definition, has a CW structure with only two~$E_\infty$ cells, minimality is now easily checked.
\end{proof}

To round off the discussion of the cotangent complex, let us also note that Basterra and Mandell obtain an equivalence
\begin{equation}\label{eq:taq_MO}
\Omega_{\bbS}(\MO)\simeq\MO\wedge\bo,
\end{equation}
of~$\MO$-modules, where~$\bo$ is the~connective cover of the real topological K-theory spectrum that has~$\Omega^\infty\bo=\BO$. This is a corollary of~\cite[Theorem 5]{Basterra+Mandell}, stated in {\it loc.cit.} for the complex bordism spectrum, but the real case is similar.


\section{Applications}\label{sec:applications}

In this section, we collect some applications, with an emphasis on the relationship to the unoriented bordism spectrum~$\MO$.


\subsection{Non-existence of~\texorpdfstring{$E_\infty$}{Einfty} maps}

The starting point for our applications is the following result.

\begin{proposition}\label{prop:Baker}
There does not exist an~$E_\infty$ map~$\rmH\bbF_2\to\MO$.
\end{proposition}

According to~\cite{Baker+Richter:Thom}, this is shown in Gilmour's thesis, generalizing an argument with power operations by Hu, Kriz, and May~\cite{Hu+Kriz+May}. See~\cite{Baker:free} for a proof. Since there does exit an~$E_\infty$ map~$\bbSt\to\MO$ by Example~\ref{ex:Eoo_maps_SSt_MO}, 
we have the following corollary.

\begin{corollary}\label{cor:no Eoo map_p=2}
There does not exist an~$E_\infty$ map~$\rmH\bbF_2\to\bbSt$.
\end{corollary}

Continuing the discussion in Example~\ref{ex:Mahowald}, the factorization of~$\rmS^1\rightarrow\rmQ(\rmS^1)$ over~$\Omega^2\rmS^3$ gives a map
\begin{equation}\label{eq:section}
\rmH\bbF_2\longrightarrow\bbSt 
\end{equation}
of commutative ring spectra up to homotopy. The existence of such a map is also clear from Theorem~\ref{thm:mod2isEM}.
The map~\eqref{eq:section} is a section of the truncation map~$\bbSt\to\rmH\bbF_2$. However, while the truncation map is~$E_\infty$, this section can not have an~$E_\infty$ representative. 

The first aim of this section is to generalize the preceding corollary to all primes.

\begin{theorem}\label{thm:no Eoo map}
There does not exist an~$E_\infty$ map~$\rmH\bbF_p\to\bbSp$.
\end{theorem}

While we have already seen this to be true in the case~$p=2$, the following proof works for all primes.

\begin{proof}
Suppose there were such a map. Then the composition
\[
\rmH\bbF_p\longrightarrow\bbSp\longrightarrow\rmH\bbF_p
\]
with the truncation to~$\rmH\pi_0(\bbSp)=\rmH\bbF_p$ would be~$E_\infty$ and an equivalence, hence an equivalence of commutative~$\bbS$-algebras. Therefore it would induce an equivalence in topological Andr\'e-Quillen homology~$\TAQ^{\bbS}(?;\bbF_2)=\Omega_{\bbS}(?)\wedge_?\rmH\bbF_2$. However, by Theorem~\ref{thm:TAQ_SSp}, we have
\[
\TAQ^{\bbS}(\bbSt;\bbF_2)=\Omega_{\bbS}(\bbSt)\wedge_{\bbSt}\rmH\bbF_2\simeq
\Sigma\bbSt\wedge_{\bbSt}\rmH\bbF_2\simeq\Sigma\rmH\bbF_2,
\]
while~$\TAQ^{\bbS}(\rmH\bbF_2;\bbF_2)$ is known to be non-trivial in other dimensions as well. In fact, it has been completely computed in unpublished work of Kriz, and Basterra-Mandell. See~\cite{Lazarev} and~\cite{Baker:free} for the precise statements.
\end{proof}

\begin{remark}\label{rem:E2}
In light of the recent interest in~$E_n$ genera~\cite{Chadwick+Mandell}, the reader may wonder if there are maps~$\rmH\bbF_p\to\bbSp$ that are somewhat compatible with the~$E_\infty$ multiplications, but not entirely so. This would be~$E_n$ maps for some integer~$n$ such that~$1<n<\infty$. And indeed there are such maps: The~$E_2$ maps~\hbox{$\Omega^2\rmS^3\to\rmQ(\rmS^1)$} over~$\rmB\GL_1(\bbS_p)$ that extend the inclusion~$\rmS^1\to\rmQ(\rmS^1)$ induce~$E_2$ maps on the level of Thom spectra, and these are~$\rmH\bbF_p$ and~$\bbSp$, respectively, again by Hopkins' extension of Mahowald's theorem and Theorems~\ref{thm:Thom_2} and~\ref{thm:Thom_odd}. In particular, there are maps~$\rmH\bbF_p\to\bbSp$ of homotopy commutative ring spectra in the traditional sense of the words. 
\end{remark}

\begin{remark}
According to the preceding Remark~\ref{rem:E2}, there is an~$E_2$ map~$\rmH\bbF_p\to\bbSp$. The composition with the~$E_\infty$ map~$\rmH\bbZ\to\rmH\bbF_p$ gives rise to an~$E_2$ map~$\rmH\bbZ\to\bbSp$. In particular, the versal examples~$\bbSp$ are~$A_\infty$ algebras under the integral Eilenberg-Mac Lane spectrum~$\rmH\bbZ$. By work of Shipley~\cite{Shipley}, there is then a differential graded algebra~$A$ (even one over~$\bbF_p$) such that~$\bbSp$ and~$\rmH A$ are equivalent as~$A_\infty$ ring spectra. The homology of~$A$ is the homotopy of~$\bbSp$. And, there is also a Quillen equivalence between the category of~$\bbSp$-module spectra and the category of differential graded modules over that same differential graded algebra~$A$. Unfortunately, the differential graded algebra~$A$ that can be derived from the general results of~\cite{Shipley}, while explicit, is everything but small.
\end{remark}


\subsection{Exotic~\texorpdfstring{$E_\infty$}{Einfty} structures on graded Eilenberg-Mac Lane spectra}

By Theorem~\ref{thm:mod2isEM},~Theorem~\ref{thm:additive_implies_multiplicative}, and~Theorem~\ref{thm:steinberger:splitting}, the commutative~$\bbS$-algebras~$\MO$,~$\bbSt$, and~$\bbSp$ are equivalent, as homotopy commutative ring spectra, to graded Eilenberg-Mac Lane spectra with the Boardman multiplication.

\begin{theorem}\label{thm:non-uniqueness-for-structures}
  The multiplications on the commutative~$\bbS$-algebras~$\MO$,~$\bbSt$, and~$\bbSp$ are not~$E_\infty$ equivalent to the Boardman multiplications.
\end{theorem}

\begin{proof}
  Otherwise, these spectra would be commutative~$\rmH\bbF_p$-algebras for suitable prime numbers~$p$, and they would receive an~$E_\infty$-structure map from~$\rmH\bbF_p$, contradicting Proposition~\ref{prop:Baker}, Corollary~\ref{cor:no Eoo map_p=2}, or Theorem~\ref{thm:no Eoo map}, respectively.
\end{proof}

We see again that, while~$E_\infty$ structures on discrete Eilenberg-Mac Lane spectra are unique, this is not the case for graded Eilenberg-Mac Lane spectra, no matter what the prime in question is. 

\begin{proposition}
For all primes~$p$, the space~$\GL_1(\bbSp)$ of units in~$\bbSp$ is not a product of Eilenberg-Mac Lane spaces.
\end{proposition}

\begin{proof}
Recall from Proposition~\ref{prop:homotopy} that the first non-trivial homotopy group of~$\bbSp$ appears in dimension~$4(p-1)$ and is isomorphic to~$\bbZ/p$. Therefore, if~$\GL_1(\bbSp)$ were a product of Eilenberg-Mac Lane spaces, then there would be a splitting
\[
\GL_1(\bbSp)\simeq \rmK(\bbZ/p,4(p-1))\times L
\]
with one of the factors the corresponding Eilenberg-Mac Lane space, such that a generator of the homotopy group~$\pi_{4(p-1)}(\bbSp)$ corresponds to a fundamental class of~$\rmK(\bbZ/p,4(p-1))$. The H-spaces structure on the space of units that comes from the multiplication on~$\bbSp$ induces an H-spaces structure on the retract~$\rmK(\bbZ/p,4(p-1))$. But, there is only one H-space structure on this Eilenberg-Mac Lane space, the standard one. It is known from Cartan's computation of the homology rings of Eilenberg-Mac Lane spectra that~\hbox{$\rmH_*(\rmK(\bbZ/p,4(p-1));\bbZ/p)$} has a divided power structure. Since the characteristic is prime, every element of positive degree has to be nilpotent. But, the composition
\[
\Sigma^\infty_+(\rmK(\bbZ/p,4(p-1)))\longrightarrow\Sigma^\infty_+(\GL_1(\bbSp))\longrightarrow\bbSp
\]
induces an isomorphism in homology in degree~\hbox{$4(p-1)$}, and it respects the multiplications. Therefore, the image of a fundamental class would be a non-trivial nilpotent element in even positive degrees, in contradiction to the fact that~$\rmH_*(\bbSp)$ is free.
\end{proof}

The analog of the preceding proposition also holds for~$\MO$, with essentially the same proof, as explained to me by Tyler Lawson.


\subsection{The structure map of unoriented bordism}

Recall, from Example~\ref{ex:Eoo_maps_SSt_MO}, that there is an~$E_\infty$ map~$s\colon\bbSt\to\MO$ which is unique up to homotopy of~$E_\infty$ maps.

\begin{proposition}\label{prop:DyerLashof}
The structure map~$s\colon\bbSt\to\MO$ is not injective in homology and homotopy.
\end{proposition}

\begin{proof}
In~\cite[IX.7.4]{lewis}, Lewis has shown that the Thom isomorphism commutes with the Araki-Kudo-Dyer-Lashof operations. This reduces the statement about the homology to the same question about the map~$\rmQ(\rmS^1)\to\BO$. The homology of~$\rmQ(\rmS^1)$ is the free algebra over the Araki-Kudo-Dyer-Lashof algebra on one generator, say~$a$, in degree~$1$, see~\cite[II.4]{cohen+lada+may}. The homology of~$\BO$ is polynomial on generators~$e_1,e_2,\dots$ with~$\dim(e_j)=j$, and the operations have been computed in~\hbox{\cite[p.~133]{Kochman}} and~\cite{Priddy}. In particular, we know that~$Q^3e_1=e_1^4$. By definition of~$s$, we also have~$s_*a=e_1$, so that the different elements~$Q^3a$ and~$a^4$ are both mapped to the same element~$e_1^4$. The proves the non-injectivity of the map in homology.

The statement for homotopy follows from immediately from the statement for homology and the fact that the homotopy embeds into the homology as the primitive elements.
\end{proof}

The preceding proof gives slightly more information: In positive dimensions, the first non-trivial element in the higher homotopy of~$\bbSt$, which lives in dimension~$4$ by Proposition~\ref{prop:homotopy}, is mapped to zero in~$\pi_4(\MO)$. 

\begin{corollary}
The structure map~$\bbSt\to\MO$ does not admit a retraction.
\end{corollary}

This is clearly true in the homotopy category spectra, and {\it a fortiori} in that of commutative~$\bbS$-algebras. In the latter, even more is true:

\begin{proposition}\label{prop:no_MO->SSp}
There does not exist an~$E_\infty$ map~$\MO\to\bbSt$.
\end{proposition}

\begin{proof}
Otherwise, the composition with the structure map~$\bbSt\to\MO$ would be an~$E_\infty$ self-map of~$\bbSt$. By Corollary~\ref{cor:id_S}, this composition would be homotopic (even as~$E_\infty$ maps) to the identity. Therefore, the hypothetical map would be a retraction for the structure map, contradicting the preceding corollary.
\end{proof}

In particular, the truncation~$\MO\to\rmH\bbF_2$ does not factor through~$\bbSt$ as an~$E_\infty$ map. Here is another consequence of the preceding proposition.

\begin{corollary}
The commutative~$\bbSt$-algebra~$\MO$ is not free (or `polynomial'), i.e.~it is not equivalent to one of the form~$\bbSt\wedge\bbP X$ for some spectrum~$X$.
\end{corollary}

\begin{proof}
Otherwise~$X\to\star$ would induce an~$E_\infty$ map
\[
\MO\simeq\bbSt\wedge\bbP X\longrightarrow\bbSt\wedge\bbP\star\simeq\bbSt\wedge\bbS\simeq\bbSt,
\]
in contradiction to Proposition~\ref{prop:no_MO->SSp}.
\end{proof}

All these results demonstrate that the picture suggested by Thom's computation
\[
\pi_*(\MO)\cong\bbF_2[x_2,x_4,x_5,x_6,x_8,\dots]
\]
of the homotopy ring is misleading when it comes to understanding~$\MO$ itself as an~$E_\infty$ ring spectrum.


\section*{Acknowledgments}

This research has been supported by the Danish National Research Foundation (DNRF) through the Centre for Symmetry and Deformation. I would also like to thank Andy Baker, Andrew Blumberg, John Francis, Mike Hopkins, Gerd Laures, Tyler Lawson, Thomas Nikolaus, Justin Noel, Birgit Richter, and Stefan Schwede for discussions related to this work, and the referees for their detailed reports.



\vfill

\parbox{\linewidth}{%
Department of Mathematical Sciences\\
University of Copenhagen\\
Universitetsparken 5\\
2100 Copenhagen \O\\
DENMARK\\
\phantom{ }\\
\href{mailto:szymik@math.ku.dk}{szymik@math.ku.dk}\\
\href{http://www.math.ku.dk/~xvd217}{www.math.ku.dk/$\sim$xvd217}
}


\begin{thebibliography}{MMMM00}

\bibitem[Ang08]{Angeltveit} V. Angeltveit. Topological Hochschild homology and cohomology of~$A_\infty$ ring spectra. Geom. Topol. 12 (2008) 987--1032.

\bibitem[ABGHR]{ABGHR}
M. Ando, A.J. Blumberg, D.J. Gepner, M.J. Hopkins, C. Rezk.
An~$\infty$-categorical approach to~$R$-line bundles,~$R$-module Thom spectra, and twisted~$R$-homology. To appear in: J. Topol. 

\bibitem[AT65]{araki+toda:1} S. Araki, H. Toda. 
Multiplicative structures in mod~$q$ cohomology theories. I. 
Osaka J. Math. 2 (1965) 71--115.

\bibitem[AT66]{araki+toda:2} S. Araki, H. Toda. 
Multiplicative structures in mod~$q$ cohomology theories. II. 
Osaka J. Math. 3 (1966) 81--120.
 
\bibitem[Ast97]{astey} L. Astey. Commutative~$2$-local ring
  spectra. Proc. Roy. Soc. Edinburgh Sect. A 127 (1997) 1--10.
  
\bibitem[Bak]{Baker:BP} A. Baker.~$BP$: Close encounters of the~$E_\infty$ kind.
To appear in:  J. Homotopy Relat. Struct.\\
\href{http://dx.doi.org/10.1007/s40062-013-0051-6}{http://dx.doi.org/10.1007/s40062-013-0051-6}
  
\bibitem[Bak']{Baker:free} A. Baker. Calculating with topological Andr\'e-Quillen theory, I: Homotopical properties of universal derivations and free commutative~$S$-algebras.\\
\href{http://arxiv.org/abs/1208.1868}{http://arxiv.org/abs/1208.1868}

\bibitem[BGR08]{Baker+Gilmour+Reinhard} A. Baker, H. Gilmour, P. Reinhard.
Topological Andr\'e-Quillen homology for cellular commutative~$S$-algebras.
Abh. Math. Semin. Univ. Hambg. 78 (2008) 27--50. 

\bibitem[BM04]{Baker+May} A. Baker, J.P. May. Minimal atomic complexes. Topology 43 (2004) 645--665.

\bibitem[BR08]{Baker+Richter:connective} A. Baker, B. Richter.
Uniqueness of~$E_\infty$ structures for connective covers.
Proc. Amer. Math. Soc. 136 (2008) 707--714.

\bibitem[BR]{Baker+Richter:Thom} A. Baker, B. Richter. Some properties of the Thom spectrum over loop suspension of complex projective space. To appear in: Contemporary Mathematics volume Proceedings of the Fourth Arolla Conference on Algebraic Topology.\\
\href{http://arxiv.org/abs/1207.4947}{http://arxiv.org/abs/1207.4947}

\bibitem[Bas99]{Basterra} M. Basterra.
Andr\'e-Quillen cohomology of commutative~$S$-algebras.
J. Pure Appl. Algebra 144 (1999) 111--143. 

\bibitem[BM05]{Basterra+Mandell} M. Basterra, M.A. Mandell. Homology and cohomology of~$E_\infty$ ring spectra. Math. Z. 249 (2005) 903--944.

\bibitem[BM72]{bass+madsen} N.A. Baas, I. Madsen.
On the realization of certain modules over the Steenrod algebra. 
Math. Scand. 31 (1972) 220--224. 

\bibitem[Blu10]{Blumberg} A.J. Blumberg. Topological Hochschild homology of Thom spectra which are~$E_\infty$ ring spectra. J. Topol. 3 (2010) 535--560.

\bibitem[Boa80]{boardman} J.M. Boardman. Graded Eilenberg-Mac Lane ring spectra. Amer. J. Math. 102 (1980) 979--1010. 

\bibitem[BMMS86]{steinberger} R.R. Bruner, J.P. May, J.E. McClure,
  M. Steinberger.~$H\sb \infty$ ring spectra and their
  applications. Lecture Notes in Mathematics 1176. Springer-Verlag,
  Berlin, 1986.
  
\bibitem[CM]{Chadwick+Mandell} S.G. Chadwick, M.A. Mandell.~$E_n$ Genera.\\
\href{http://arxiv.org/abs/1310.3336}{http://arxiv.org/abs/1310.3336}

\bibitem[CLM76]{cohen+lada+may} F.R. Cohen, T.J. Lada, J.P. May. The
  homology of iterated loop spaces. Lecture Notes in Mathematics
  533. Springer-Verlag, Berlin-New York, 1976.
  
\bibitem[EKMM97]{EKMM} A.D. Elmendorf, I. Kriz, M.A. Mandell, J.P. May. Rings, modules, and algebras in stable homotopy theory. With an appendix by M. Cole. Mathematical Surveys and Monographs, 47. American Mathematical Society, Providence, RI, 1997.
  
\bibitem[HKM01]{Hu+Kriz+May} P. Hu, I. Kriz, J.P. May. Cores of spaces, spectra, and~$E_\infty$ ring spectra. Equivariant stable homotopy theory and related areas (Stanford, CA, 2000). Homology Homotopy Appl. 3 (2001) 341--354.

\bibitem[Koc73]{Kochman}  S.O. Kochman. Homology of the classical groups over the Dyer-Lashof algebra. Trans. Amer. Math. Soc. 185 (1973) 83--136.

\bibitem[Kuh06]{Kuhn} N.J. Kuhn. Localization of Andr\'e-Quillen-Goodwillie towers, and the periodic homology of infinite loopspaces. Adv. Math. 201 (2006) 318--378.

\bibitem[Laz04]{Lazarev} A. Lazarev. Cohomology theories for highly structured ring spectra. Structured ring spectra, 201--231, London Math. Soc. Lecture Note Ser., 315, Cambridge Univ. Press, Cambridge, 2004.

\bibitem[LMS86]{lewis} L.G. Lewis, J.P. May, M. Steinberger,
  J.E. McClure. Equivariant stable homotopy theory. Lecture Notes in
  Mathematics 1213. Springer-Verlag, Berlin, 1986.

\bibitem[Mah79]{mahowald} M. Mahowald. Ring spectra which are Thom
  complexes. Duke Math. J. 46 (1979) 549--559.
  
\bibitem[MRS01]{Mahowald+Ravenel+Shick} M. Mahowald, D.C. Ravenel, P. Shick.
The Thomified Eilenberg-Moore spectral sequence. 
Cohomological methods in homotopy theory (Bellaterra, 1998), 249--262. Progr. Math., 196, Birkh\"auser, Basel, 2001. 
  
\bibitem[MMSS01]{MMSS}  M.A. Mandell, J. P. May, S. Schwede, B. Shipley.
Model categories of diagram spectra.
Proc. London Math. Soc. 82 (2001) 441--512.

\bibitem[PR85]{pazhitnov+rudyak} A.V. Pazhitnov, Y.B. Rudyak.
Commutative ring spectra of characteristic~$2$.
Math. USSR-Sb. 52 (1985) 471--479.

\bibitem[Pri75]{Priddy} S. Priddy. Dyer-Lashof operations for the classifying spaces of certain matrix groups. Quart. J. Math. Oxford Ser. 26 (1975) 179--193.

\bibitem[Ric03]{Richter:Symmetry} B. Richter. Symmetry properties of the Dold-Kan correspondence. Math. Proc. Cambridge Philos. Soc. 134 (2003) 95--102.

\bibitem[Ric06]{Richter:Homotopy} B. Richter. Homotopy algebras and the inverse of the normalization functor. J. Pure Appl. Algebra 206 (2006) 277--321.

\bibitem[Rog08]{rognes:galois} J. Rognes. 
Galois extensions of structured ring spectra. Stably dualizable groups.
Mem. Amer. Math. Soc. 192 (2008) 898.

\bibitem[Rud98]{rudyak} Y.B. Rudyak. On Thom spectra, orientability, and cobordism. Springer-Verlag, Berlin, 1998.
  
\bibitem[SVW98]{SVW} R. Schw\"anzl, R.M. Vogt, F. Waldhausen. Adjoining roots of unity to~$E_\infty$ ring spectra in good cases---a remark. Homotopy invariant algebraic structures (Baltimore, MD, 1998) 245--249. Contemp. Math. 239, Amer. Math. Soc., Providence, RI, 1999.

\bibitem[Shi07]{Shipley} B. Shipley.~$\rmH\bbZ$-algebra spectra are differential graded algebras. Amer. J. Math. 129 (2007) 351--379.

\bibitem[Szy]{Szymik} M. Szymik. String bordism and chromatic characteristics.\\
\href{http://arxiv.org/abs/1312.4658}{http://arxiv.org/abs/1312.4658}

\bibitem[W{\"u}r86]{würgler} U. W\"urgler. Commutative ring-spectra of
  characteristic~$2$. Comment. Math. Helv. 61 (1986) 33--45.
  
\bibitem[Yan92]{yan:proc} D.Y. Yan. On the Thom spectra over~$\Omega(\mathrm{SU}(n)/\mathrm{SO}(n))$ and Maho\-wald's~$X_k$ spectra. Proc. Amer. Math. Soc. 116 (1992) 567--573.
  
\bibitem[Yan94]{yan:trans} D.Y. Yan. The Brown-Peterson homology of Mahowald's~$X_k$ spectra. Trans. Amer. Math. Soc. 344 (1994) 261--289.

\end{thebibliography}
\end{document}